\documentclass[12pt, a4paper]{article}

\usepackage{amssymb,amsfonts}

\usepackage{mathtools}

\usepackage{amsmath,amsthm}
\usepackage{mathrsfs}

\usepackage{hyperref}

\usepackage{exscale,txfonts}

\usepackage[all]{xy}

\usepackage[active]{srcltx}

\usepackage{epic,eepic,eepicemu}
\usepackage[usenames]{color}

\usepackage[normalem]{ulem}

\usepackage{enumerate}

\allowdisplaybreaks

\numberwithin{equation}{section}

\definecolor{dgreen}{rgb}{0.1,0.7,0.1}
\definecolor{purple}{rgb}{0.49, 0.06, 0.51}

\def\N{\mathbb{N}}

\def\Z{\mathbb{Z}}

\def\la{\langle}
\def\ra{\rangle}
\def\s{\sigma}

\newcommand{\vt}{\vartheta}
\newcommand{\ox}{\otimes}
\newcommand{\x}{\times}
\newcommand{\sm}{\setminus}
\newcommand{\Nil}{\mathrm{Nil}}
\newcommand{\wt}{\widetilde}
\newcommand{\id}{\mathrm{id}}
\newcommand{\vf}{\varphi}

\newcommand{\ovl}{\overline}

\newcommand{\too}{\longrightarrow}

\newcommand{\bbar}{\overline{\phantom{x}}}

\DeclareMathOperator{\sign}{sign}

\DeclareMathOperator{\Sym}{Sym}
\DeclareMathOperator{\rk}{rk}

\DeclareMathOperator{\Int}{Int}

\DeclareMathOperator{\Trd}{Trd}

\DeclareMathOperator{\diag}{diag}

\DeclareMathOperator{\End}{End}

\DeclareMathOperator{\Herm}{\mathfrak{Herm}}

\renewcommand{\geq}{\geqslant}
\renewcommand{\leq}{\leqslant}

\renewcommand{\le}{\leqslant}

\newcommand{\df}{\emph}
\newcommand{\ve}{\varepsilon}

\newcommand{\ad}{\mathrm{ad}}

\newcommand{\tas}{T_{(A,\s)}}
\newcommand{\tasu}{T_{(A,\s,u)}}

\newcommand{\das}{D_{(A,\s)}}

\newcommand{\sas}{\Sym(A,\s)}

\newcommand{\pf}[1]{\langle\!\langle #1\rangle\!\rangle}
\newcommand{\qf}[1]{\langle #1\rangle}

\newcommand{\qnd}[1]{\langle #1\rangle^{\ns}}

\newcommand{\ns}{\mathrm{ns}}

\renewcommand{\iff}{\Leftrightarrow}

\newtheorem{thm}{Theorem}[section]
\newtheorem{prop}[thm]{Proposition}
\newtheorem{cor}[thm]{Corollary}
\newtheorem{lemma}[thm]{Lemma}
\theoremstyle{definition}
\newtheorem{defi}[thm]{Definition}

\newtheorem{remark}[thm]{Remark}

\begin{document}
\title{Signatures of hermitian forms, positivity, and an answer to a question of Procesi and Schacher}
\author{Vincent Astier and Thomas Unger}

\date{}
\maketitle

\begin{abstract}
Using the theory of signatures of hermitian forms over algebras with involution, developed by us in earlier work, we 
introduce a notion of positivity for symmetric elements and prove a noncommutative analogue of Artin's solution to 
Hilbert's 17th problem, characterizing totally positive elements in terms of weighted sums of hermitian squares. As a consequence 
we obtain an earlier result of  Procesi and Schacher  and give a complete answer to their question about 
representation of elements as sums of hermitian squares.
\smallskip

\smallskip


\end{abstract}


\section{Introduction}

We use the theory of signatures of hermitian forms,
a tool  we developed and studied in \cite{A-U-Kneb} and \cite{A-U-prime}, 
to introduce a natural notion of positivity for symmetric
elements in an algebra with involution, inspired by the theory of quadratic forms; signatures of one-dimensional 
hermitian forms over algebras with an involution can take values outside of $\{-1,1\}$ and it is therefore natural to single out those symmetric elements whose associated hermitian form has maximal signature at a given ordering. We call such elements maximal at the ordering
and characterize the elements that are maximal at all orderings in terms of weighted
sums of hermitian squares, thus obtaining an analogue of Artin's solution to Hilbert's 17th problem for algebras with 
involution, cf. Section~\ref{sec:max}.   
The proof is obtained via signatures, allowing us to use the hermitian version of Pfister's local-global principle. 
This provides a short and conceptual argument, based on torsion in the Witt group.

Procesi and Schacher \cite{P-S} already considered  such a noncommutative version of Artin's theorem in this context,  using a notion of positivity based on involution trace forms   which goes back to Weil \cite{Weil}.
They showed  that every totally positive element (in their sense) in an algebra with involution is a sum of 
squares of symmetric elements, and thus of hermitian squares, with weights, cf. \cite[Theorem~5.4]{P-S}. 
They also asked  if these weights could be removed \cite[p.~404]{P-S}. 
The answer to this question is in general no, as shown in \cite{K-U}.

Our approach via signatures makes it possible to obtain the sum of hermitian squares version of their theorem as a consequence of  Theorem~\ref{main_thm_2}. It also allows us to single out the set of orderings relevant to their question (the non-nil orderings) and to   rephrase it in a natural way, which can then be fully answered (Theorem~\ref{biscuit}).

\section{Algebras with involution and signatures of hermitian forms}

We present the notation and main tools  used in this paper and refer to the standard references 
\cite{Knus}, \cite{BOI}, \cite{Lam} and \cite{Sch} as well as \cite{A-U-Kneb} and \cite{A-U-prime}
for the details.

\subsection{Algebras with involution, hermitian forms}

For a ring $A$, an involution $\s$ on $A$ and $\ve\in\{-1,1\}$, we denote
the set of $\ve$-symmetric elements of $A$ with respect to $\s$ by $\Sym_\ve(A,\s) =\{a\in A\mid \s(a)=\ve a\}$. We also denote the set of invertible elements of $A$  by $A^\x$ and let 
 $\Sym_\ve(A,\s)^\x:= \Sym_\ve(A,\s) \cap A^\x$.

Let $F$ be a field of characteristic different from $2$. We denote by $W(F)$ the Witt ring of $F$, by $X_F$  the space of orderings of $F$,
and by $F_P$ a real closure of $F$ at an ordering $P\in X_F$. 
We allow for the possibility that $F$ is not formally real, i.e. that $X_F=\varnothing$.
By an 
\emph{$F$-algebra with involution} we mean a pair $(A,\s)$ where $A$ is a 
finite-dimensional simple $F$-algebra with centre a field  $K$, equipped with an involution  $\s:A\to A$, such that $F = K \cap \Sym(A,\s)$.  Observe that $\dim_F K 
\le 2$.
We say that 
$\s$ is \df{of the first kind} if $K=F$ and \df{of the second kind} otherwise. 
We let $\iota=\s|_{K}$ and note that $\iota =\id_F$ if $\s$ is of the first kind.
If $A$ is a division algebra, we call $(A,\s)$ an \df{$F$-division algebra with involution}.

Let $(A, \s)$ be an $F$-algebra with involution.
It follows from the structure theory of $F$-algebras with involution  that $A$ is isomorphic to a full matrix algebra $M_\ell(D)$ for a unique 
$\ell\in \N$ and an $F$-division algebra $D$
(unique up to isomorphism) which is equipped with an involution $\vt$ of the same kind as $\s$, cf. \cite[Thm.~3.1]{BOI}. 
For $B=(b_{ij})\in M_\ell(D)$ we let $\vt^t(B)=(\vt(b_{ji}))$.
We denote Brauer equivalence by~$\sim$, isomorphism by $\cong$ and isometry of forms by $\simeq$.

For $\ve\in\{-1,1\}$ we write $W_\ve(A,\s)$ for the \df{Witt group} of Witt equivalence classes of nonsingular $\ve$-hermitian forms, defined on 
finitely generated right $A$-modules. Note that $W_\ve(A,\s)$ is a $W(F)$-module.
For a nonsingular $\ve$-hermitian form $h$ over $(A,\s)$ the notation $h\in W_\ve(A,\s)$ signifies that $h$ is identified with its Witt class in $W_\ve(A,\s)$.

For $a_1, \ldots, a_k \in F$ the notation $\qf{a_1,\ldots, a_k}$ stands for the quadratic form $(x_1,\ldots, x_k) \in F^k \mapsto  \sum_{i=1}^k a_i x_i^2 \in F$, as usual, whereas for $a_1, \ldots, a_k$ in $\Sym_\ve(A,\s)$ the notation $\qf{a_1,\ldots, a_k}_\s$ stands for the diagonal $\ve$-hermitian form  
\[ \bigl(  (x_1,\ldots, x_k), (y_1,\ldots, y_k)  \bigr) \in A^k \x A^k \mapsto   \sum_{i=1}^k \s(x_i) a_i y_i \in A.\]
In each case, we call $k$ the \emph{dimension} of the form. 

In this paper, we are mostly interested in hermitian forms ($\ve=1$) and only occasionally in skew-hermitian forms ($
\ve=-1$). When $\ve=1$, we write $\Sym(A,\s)$  and  $W(A,\s)$ instead of $\Sym_{1}(A,\s)$  and  $W_{1}(A,\s)$, respectively.

Let $h:M\x M\to A$ be a hermitian form over $(A,\s)$. We sometimes write $(M,h)$ instead of $h$.
The \emph{rank} of $h$, $\rk(h)$, is the rank of the $A$-module $M$.
The set of elements represented by $h$ is denoted by
\[\das(h): =\{ u \in \Sym(A,\s) \mid  \exists x\in M\text{ such that } h(x,x)=u\}.\]

We denote by $\Int(u)$ the inner automorphism determined by $u \in A^\x$, where  $\Int(u)(x):= uxu^{-1}$ for $x\in A$. 

\begin{remark}
If  $F$ is not formally real, many results in this paper are trivially true since $W(A,\s)$ is torsion in this case (see \cite[Theorem~4.1]{LU1} and note that this theorem, being a reformulation of \cite[Theorem~3.2]{LU1}, is actually valid for any field of characteristic not~$2$).
\end{remark}

\subsection{Morita theory}\label{Morita} 

For the remainder of the paper we fix some   field $F$ of characteristic not $2$ and some $F$-algebra with involution $(A,\s)$, 
where $\dim_K A =m=n^2$ and $A\cong M_\ell(D)$ for some $F$-division algebra $D$ which is equipped with an involution $\vt$ of the same kind as $\s$. Recall that the integer $n$ is called the \emph{degree} of $A$, $\deg A$.

By \cite[4.A]{BOI}, there exists $\ve \in \{-1,1\}$ and an invertible matrix $\Phi \in M_\ell(D)$ such that 
$\vt(\Phi)^t=\ve \Phi$ and 
$(A,\s)\cong (M_{\ell}(D), \ad_\Phi)$,
where $\ad_\Phi = \Int(\Phi) \circ \vt^t$. (In fact, $\Phi$ is the Gram matrix of an $\ve$-hermitian form over $(D,\vt)$.)
Note that $\ad_\Phi=\ad_{\lambda\Phi}$ for all $\lambda \in F^\x$ and that $\ve=1$ when $\s$ and $\vt$ are of the same type.  
We fix an isomorphism of $F$-algebras with involution
 $f:(A,\s)\to (M_{\ell}(D), \ad_\Phi)$.

\begin{lemma}\label{eps} 
We may choose $\vt$ above such that $\ve=1$, except when $A\cong M_\ell(F)$ with $\ell$ even and $\s$ symplectic, in 
which case $(D, \vt, \ve)= (F, \id_F, -1)$. 
\end{lemma}

\begin{proof} We consider all possible cases, with reference to \cite[Corollary~2.8]{BOI} for involutions of the first kind.

Case 1:  $\s$, and thus $\vt$, of the second kind. In this case, if $\ve=-1$,  let $u\in  K^\x$ be such that $\vt(u)=-u$ and replace $\vt$ by $\Int(u)\circ \vt$ and $\Phi$ by $u\Phi$.

Case 2: $\s$, and thus $\vt$, of the first kind and $\deg D$ even. Then $D$ can be equipped with both orthogonal and symplectic involutions and so we may choose     $\vt$ to be of the same type as $\s$ so that  
$\vt(\Phi)^t= \Phi$.

Case 3: $\s$, and thus $\vt$, of the first kind, $\deg D$ odd and $\deg A$ also odd. In this case, $D=F$, $\vt=\id_F$, 
$A$ is split (i.e. $A\sim F$) and $\s$ must be orthogonal. Thus $\ve=1$ since  $\vt$ and $\s$ are both orthogonal.

Case 4: $\s$, and thus $\vt$, of the first kind, $\deg D$ odd and $\deg A$ even. In this case, $D=F$, $\vt=\id_F$ and $A$ is split. If $\s$ is orthogonal, then $\ve=1$ since  $\vt$ and $\s$ are both orthogonal. If $\s$ is symplectic, then 
$\ve=-1$.
\end{proof}

Given an $F$-algebra with involution $(B,\tau)$ we denote by 
$\Herm_\ve(B,\tau)$ the category of $\ve$-hermitian 
 forms over $(B, \tau)$ (possibly singular), cf. \cite[p.~12]{Knus}. The isomorphism $f$ trivially induces an equivalence of categories
 $f_*: \Herm(A,\s) \too \Herm(M_{\ell}(D), \ad_\Phi)$. Furthermore, the $F$-algebras with involution $(A,\s)$ and
 $(D,\vt)$ are Morita equivalent, cf. \cite[Chapter~I, Theorem~9.3.5]{Knus}. In this paper we make repeated use of a particular Morita 
 equivalence  between $(A,\s)$ and $(D,\vt)$, following the  approach in \cite{LU2} (see also \cite[\S2.4]{A-U-Kneb} 
 for the case of nonsingular forms and \cite[Proposition~3.4]{A-U-Kneb} for a justification of why
 using this equivalence is  as good as using any other equivalence), namely:
\begin{equation}\label{diagram}
\xymatrix{
\Herm(A,\s)\ar[r]^--{f_*} &    \Herm(M_{\ell}(D),\ad_\Phi)\ar[r]^--{s}  &  \Herm_\ve(M_\ell(D), \vt^t) 
\ar[r]^--{g} & \Herm_\ve(D,\vt),}
\end{equation}
where  $s$ is the \emph{scaling by $\Phi^{-1}$} Morita equivalence, given by $(M,h)\mapsto (M, \Phi^{-1}h)$  
and
$g$ is the \emph{collapsing} Morita equivalence, given by $(M,h)\mapsto (D^k,b)$, 
where $k$ is the rank of $M$ as $M_\ell(D)$-module. Under the isomorphism $M \cong (D^\ell)^k$, $h$ can be identified with the form  $(M_{k,\ell}(D), \qf{B}_{\vt^t})$ for some matrix $B\in M_k(D)$ that satisfies $\vt^t(B)=\ve B$ and we take 
for $b$ the $\ve$-hermitian form whose Gram matrix is $B$. Note that $\qf{B}_{\vt^t}(X,Y):= \vt(X)^t B Y$ for all $X,Y \in M_{k,\ell}(D)$.

\subsection{Signatures of hermitian forms}\label{sec:sign}

We defined signatures of nonsingular hermitian forms over $(A,\s)$ 
 in \cite{A-U-Kneb}, inspired by \cite{BP2}, and gave a more concise presentation  in \cite[\S 2]{A-U-prime}, which we will follow in this section and to which we refer for the details. (We called them $H$-signatures in  \cite{A-U-Kneb} and 
 \cite{A-U-prime} to differentiate them from the signatures in  \cite{BP2}.)

Let $P\in X_F$ and
consider  the sequence of group morphisms (cf. \cite[Diagram~(1)]{A-U-prime})
\begin{equation}\label{seq}
\xymatrix{
W(A,\s) \ar[r]^--{r_P} & W(A\ox_F F_P, \s\ox\id) \ar[r]^--{\mu_P}_--{\cong} & W_{\ve_P}(D_P,\vt_P)\ar[r]^--{\sign_P}  & \Z,
}
\end{equation}
where $r_P$ is induced by the canonical extension of scalars map, $A\ox_F F_P$ is a matrix algebra over $D_P$,
$\vt_P$ is an involution on $D_P$, $\mu_P$ is an isomorphism induced by Morita equivalence (for example, the isomorphism induced by 
\eqref{diagram} with $(A\ox_F F_P, \s\ox\id)$ in the role of $(A,\s)$)
and  $\sign_P$ is zero if $\ve_P=-1$ and the 
Sylvester signature  at the unique ordering of $F_P$, otherwise (in which case $(D_P,\vt_P)$ is one of 
$(F_P, \id_{F_P})$, $(F_P(\sqrt{-1}), \bbar)$ or $((-1,-1)_{F_P}, \bbar)$, where $\bbar$ denotes conjugation).

Diagram~\eqref{seq} defines a morphism of groups $s_{\mu_P}: W(A,\s) \to \Z$. The map $\mu_P$ is not canonical
and a different choice may at most result in multiplying $s_{\mu_P}$ by $-1$. We define the set of 
\emph{nil-orderings} of $(A,\s)$ as follows:
\[\Nil[A,\s]:=\{ P\in X_F \mid s_{\mu_P} = 0 \}\]
and note that it does not depend on the choice of $\mu_P$, but only on the Brauer class of $A$ and the type of $\s$. 
For convenience we also introduce 
\[\wt X_F:=X_F \sm \Nil[A,\s],\]
which does not indicate the dependence on $(A,\s)$ in order to avoid cumbersome notation.

Given $P \in X_F$, we define $\sign_P^\eta$, 
the \emph{signature} at $P$ of  nonsingular hermitian forms  over $(A,\s)$,
  as follows (see also \cite{A-U-Kneb} and \cite{A-U-prime}):
\begin{enumerate}[(i)]
\item if $P\in \Nil[A,\s]$, we let $\sign_P^\eta  =0$;
\item if $P \in \wt X_F$, $\sign_P^\eta$ will be either $s_{\mu_P}$ or $-s_{\mu_P}$.  In \cite[Theorem~6.4]{A-U-Kneb} we proved that there exists a finite
tuple  $\eta=(\eta_1,\ldots, \eta_t)$ of nonsingular hermitian forms (which can all be chosen to be diagonal of dimension $1$) such that for every $Q \in \wt X_F$, $s_{\mu_Q} ( \eta) \not = (0,\ldots,0)$. 
Using $\eta$ as provided by this theorem,  let $i$ be the least integer such that $ s_{\mu_P} (\eta_i)\not=0$. We choose $\sign_P^\eta \in \{-s_{\mu_P}, s_{\mu_P}\}$ such that $\sign_P^\eta \eta_i >0$.
\end{enumerate}
In \cite[Proposition~3.2]{A-U-prime} we showed that the tuple $\eta$ (called a \emph{tuple of reference forms for 
$(A,\s)$})
can be replaced by a single diagonal hermitian form (called a \emph{reference form for $(A,\s)$}) 
which may have dimension greater than one. 

\begin{remark}\label{ricola}
If $\eta=(\eta_1,\ldots, \eta_t)$ is a tuple of reference forms for $(A,\s)$, then $\eta'=(\qf{1}_\s, \eta_1,\ldots, \eta_t)$ is also a tuple of reference forms, with the property that if $s_{\mu_P}\qf{1}_\s\not=0$, then $\sign_P^{\eta'} \qf{1}_\s >0$.
More generally, for every hermitian form $\eta_0$ over $(A,\s)$, the tuple $(\eta_0, \eta_1,\ldots, \eta_t)$ will also be a tuple of reference forms.
\end{remark}

\begin{remark} Let $(A,\s)$ and $(B,\tau)$ be Morita equivalent $F$-algebras with involution. Denoting this equivalence by 
$\mu$ and letting $\eta=(\eta_1,\ldots, \eta_t)$ be a tuple of reference forms for $(A,\s)$, it follows from \cite[Theorem~4.2]{A-U-prime} that $(\mu(\eta_1), \ldots, \mu(\eta_t))$ is a tuple of reference forms for  $(B,\tau)$.
\end{remark}

\begin{lemma}\label{rescue}
If $(D,\vt,\ve)=(F, \id_F, -1)$, then $\wt X_F=\varnothing$. 
\end{lemma}

\begin{proof} Using the notation from Section~\ref{Morita}, we have $(A,\s) \cong (M_\ell(F), \ad_\Phi)$, where $\Phi$ is a skew-symmetric matrix over $F$. Let $P\in X_F$. Then $(A \ox_F F_P ,\s\ox \id) \cong (M_\ell(F_P), \ad_{\Phi \ox \id})$ and so
$W(M_\ell(F_P), \ad_{\Phi \ox \id}) \cong W_{-1}(F_P, \id_{F_P})$ by \eqref{seq}. It follows that $\ve_P=-1$ 
in \eqref{seq} and so $P \in \Nil[A,\s]$.
\end{proof}

Use of the notation $\sign_P^\eta h$  assumes that $\eta$ is some tuple of reference forms for $(A,\s)$ 
and that $h$ is a nonsingular hermitian form over $(A,\s)$. Also, if $F$ has only one ordering $P$, we write $\sign^\eta$ instead
of $\sign_P^\eta$.

\subsection{The nonsingular part of a hermitian form}\label{nons}

Let $u$ be an  element in $\Sym(A,\s)$, not necessarily invertible. In the next sections we examine the ``positivity'' of $u$
and its relation to sums of hermitian squares
 in terms of the associated hermitian form $\qf{u}_\s$ over $(A,\s)$, which may be singular. The properties that we are interested in only depend on the nonsingular part of $\qf{u}_\s$, which motivates the remainder of this section.

We start with two lemmas, corresponding to \cite[Chapter~I, Lemma~6.2.3]{Knus} and \cite[Chapter~I, Proposition~6.2.4]{Knus}, but stated for 
possibly singular $\ve$-hermitian forms.  

\begin{lemma}\label{knus} 
Let $(D,\vt)$ be an $F$-division algebra with involution and let $(M,h)$ be an $\ve$-hermitian form over $(D,\vt)$, where $\ve \in \{-1,1\}$.  Assume that $h(x,x)=0$ for all $x\in M$. Then 
\[h=0 \quad \text { or }\quad  (D,\vt, \ve)=(F,\id_F, -1).\]
\end{lemma}

\begin{proof} Assume $h\not=0$ and let $x,z\in M$ be such that $h(x,z)=\alpha\not=0$. Let $d\in D^\x$ and let 
$y=z\alpha^{-1} d$. Then $h(x,y)=d$, and the proof proceeds as in the proof of \cite[Chapter~I, Lemma~6.2.3]{Knus}: assuming that
$\vt$ is nontrivial, we reach a contradiction and the rest of the lemma follows.
\end{proof}

\begin{lemma}\label{knus2}  
Let $(D,\vt)$ be an $F$-division algebra with involution and let $(M,h)$ be an $\ve$-hermitian form over $(D,\vt)$, where $\ve \in \{-1,1\}$. Assume that the Gram matrix of $h$ is $H$. Then there exists an invertible matrix $G \in M_\ell(D)$
such that 
\[ \vt(G)^t H G = \diag( u_1,\ldots, u_{k}, 0,\ldots,0),\]
where $u_1,\ldots, u_k \in \Sym(D,\vt)^\x$, except when $(D,\vt, \ve)=(F,\id_F, -1)$, in which case they are elements of
$\Sym_{-1}(M_2(F), {}^t)^{\x}$. 
\end{lemma}

\begin{proof} Assume first that $(D,\vt,\ve)\not=(F,\id_F, -1)$. If $h=0$, there is nothing to prove. Otherwise, 
there exists $x\in M$ such that $h(x,x)\not=0$, by Lemma~\ref{knus}.  
Then $M= xD \oplus (xD)^\perp$ and the result follows by induction.

Finally, if $(D,\vt,\ve)=(F,\id_F, -1)$, the result is well-known. 
\end{proof}

Let $(A,\s)$ be an $F$-algebra with involution and fix an isomorphism $f: (A,\s) \to (M_\ell(D), \Int(\Phi)\circ \vt^t)$ as at the
start of Section~\ref{Morita}. 
Let $u \in \Sym(A,\s)$. Since $\Phi^{-1} f(u) \in \Sym_\ve(M_\ell(D),  \vt^t)$, it is the Gram matrix of an $\ve$-hermitian form over $(D,\vt)$ and thus,  by Lemma~\ref{knus2}, 
there exists an invertible matrix $G \in M_\ell(D)$ such that 
\begin{equation}\label{iso}
 \vt(G)^t  (\Phi^{-1} f(u)) G = \diag( u_1,\ldots, u_{k}, 0,\ldots,0),
\end{equation} 
where 
$u_1,\ldots, u_{k}$ are as in Lemma~\ref{knus2}.
For $i=1,\ldots, k$, let $\vf_i$ denote the $\ve$-hermitian form over $(D,\vt)$ with Gram matrix $u_i$. 

The $F$-algebras with involution $(A,\s)$ and $(D, \vt)$ are Morita equivalent, cf. \cite[Chapter~I, Theorem~9.3.5]{Knus}.
Consider the hermitian form $\qf{u}_\s$ over $(A,\s)$.  Under the equivalences depicted in \eqref{diagram}, $\qf{u}_\s$
corresponds to the scaled 
$\ve$-hermitian form $ \qf{\Phi^{-1} f(u)}_{\vt^t}$ over  $(M_\ell(D), \vt^t)$, which then corresponds to the collapsed 
$\ell$-dimensional $\ve$-hermitian form $\vf$
with Gram matrix $ \diag(u_1,\ldots, u_{k}, 0,\ldots,0)$. Note that
\[\vf = \vf_1 \perp \ldots \perp \vf_{k} \perp  0 \perp \ldots \perp 0.\]

For $i\in \{1,\ldots, k\}$, the preimage of $\vf_i$ under these equivalences  is a nonsingular hermitian form over $(A,\s)$ which we denote by $h_i$.  Consequently we obtain the orthogonal decomposition
\[ \qf{u}_\s \simeq    h_1 \perp \ldots \perp h_{k} \perp 0 \perp \ldots \perp 0,\]
where $0$ denotes the zero form of rank $1$ over $(A,\s)$. The form 
$   h_1 \perp \ldots \perp h_{k}$ is nonsingular and we denote it by $\qnd{u}_\s$. Note that a standard argument shows 
that $\qnd{u}_\s$ is uniquely determined by $\qf{u}_\s$ up to isometry.

More generally, let $h$ be a (not necessarily diagonal) hermitian form over $(A,\s)$. By the same reasoning as above there exists a nonsingular hermitian form $h^\ns$  (also uniquely determined by $h$ up to isometry)
such that
\[h \simeq h^\ns \perp 0,\]
where $0$ is the zero form over $(A,\s)$ of suitable rank.

The following result characterizes the representation of not necessarily invertible elements in $\Sym(A,\s)$ in terms of hermitian forms.

\begin{prop}\label{three} 
Let $h$ be a hermitian form over $(A,\s)$ and let $u\in \Sym(A,\s)$. The following statements are equivalent:
\begin{enumerate}[$(i)$]
\item $u \in \das (2^r \x h)$ for some $r \in \N$.
\item The form $\qnd{u}_\s$ is a subform of $2^{r'}\x h$ for some $r' \in \N$.
\end{enumerate}
\end{prop}

\begin{proof} We use the notation from the beginning of this section and denote being a subform by $\leq$. 
Assume first that $(D,\vt,\ve)\not= (F,\id_F, -1)$.
With reference to the equivalences in \eqref{diagram},
we have the following equivalent statements (with justifications below):
\begin{align}
\exists r &\in \N\hspace{.8em}  u \in \das (2^r \x h)  \nonumber \\
&\iff  \exists r\in \N\hspace{.8em}  \Phi^{-1}f(u) \in D_{(M_\ell(D), \vt^t)} (2^r\x \Phi^{-1} f_*(h))  \label{eq1}   \\
&\iff \exists r\in \N\hspace{.8em} \vt(G)^t  (\Phi^{-1} f(u)) G = \diag(u_1,\ldots, u_k, 0,\ldots,0)\nonumber \\
&\qquad\qquad\qquad\qquad\qquad\qquad\qquad\qquad  \in D_{(M_\ell(D), \vt^t)} (2^r\x \Phi^{-1} f_*(h)) \label{eq2p} \\
&\iff \exists s\in \N\, \forall i=1,\ldots,k \hspace{.8em}  \diag(u_i, \ldots, u_i ) 
\in D_{(M_\ell(D), \vt^t)} (2^{s}\x \Phi^{-1} f_*(h)) \label{eq2} \\
&\iff \exists s\in \N\, \forall i=1,\ldots,k \hspace{.8em}  \qf{\diag(u_i, \ldots, u_i )}_{\vt^t}   \leq  2^{s} \x \Phi^{-1} f_*(h)
\nonumber \\
&\iff \exists s\in \N\hspace{.8em} \ell\x \qf{u_1}_\vt,\ldots, \ell\x \qf{u_k}_\vt \leq 2^{s}\x g(\Phi^{-1} f_*(h)) \label{eq3} \\
&\iff \exists s_1\in \N\hspace{.8em}  \qf{u_1}_\vt, \ldots, \qf{u_k}_\vt \leq 2^{s_1}\x g(\Phi^{-1} f_*(h)) \nonumber \\
&\iff \exists s_2\in \N\hspace{.8em} \qf{u_1}_\vt \perp \ldots \perp \qf{u_k}_\vt  \leq  2^{s_2}\x g(\Phi^{-1} f_*(h)) \nonumber \\
&\iff \exists r'\in \N\hspace{.8em} \qnd{u}_\s = h_1\perp\ldots \perp h_k \leq  2^{r'}\x h. \label{eq4}
\end{align}
The justifications are as follows: \eqref{eq1} follows by scaling, \eqref{eq3} follows  by collapsing 
and \eqref{eq4} follows by the full sequence of equivalences in \eqref{diagram} 
(between $(D,\vt)$ and $(A,\s)$) and the observations preceding the proposition.  
Both directions of \eqref{eq2} follow  by applying sufficiently many transformations of the form $X \mapsto \vt(Q)^t X Q$ to $\diag(u_1,\ldots, u_k, 0,\ldots,0)$ or $u_1 I_\ell, \ldots, u_k I_\ell$, where $Q$ is 
\[\diag(0,\ldots,0,1,0,\ldots,0) \qquad \text{(where $1$ can be in any position)}\]
or a permutation matrix, and summing the results. 

Finally, if $(D,\vt,\ve)= (F,\id_F, -1)$, the same argument works mutatis mutandis, using $u_i \in \Sym_{-1}(M_2(D), \vt^t)^\x$,
noting that the step from \eqref{eq2p} to \eqref{eq2} works since $\ell$ is even (indeed, $\Phi$ is an invertible 
skew-symmetric 
matrix over $F$ in the case under consideration, and is thus of even dimension).
\end{proof}

\section{Maximal elements and sums of hermitian squares}\label{sec:max}

In contrast to quadratic forms, the signature of nonsingular hermitian forms of dimension one can take more than just two values. It is therefore natural to single out those elements $u$ in $\Sym(A,\s)$ whose associated hermitian form $\qf{u}_\s$ has maximal possible signature, leading to a natural notion of positivity, which we call $\eta$-maximality (where $\eta$
is a tuple of reference forms for $(A,\s)$), cf. Definition~\ref{def:max}.

Our main result, Theorem~\ref{main_thm_2}, shows that, as in the quadratic forms case, 
Pfister's local-global principle can be used to characterize ``totally positive'' elements in terms of 
(weighted) sums of hermitian squares, providing an extension of Artin's result to  algebras with involution. 

We treat the case of invertible elements first in Theorem~\ref{main_thm_1} since its proof is more streamlined
and the arguments appear more clearly.

\begin{defi}\label{def:max} 
Let $P\in X_F$ and let $\eta$ be a tuple of reference forms for $(A,\s)$.
\begin{enumerate}[$(i)$]

\item Let
\[m_P:= \max \{\sign_P^\eta \qf{a}_\s \mid a \in \sas^\x\}.\]
We call $u \in \sas^\x$ \emph{$\eta$-maximal at $P$} if $\sign_P^\eta \qf{u}_\s=m_P$.   

\item We call a  nonsingular 
hermitian form $h$ of rank $k$ over $(A,\s)$ 
\emph{$\eta$-maximal at $P$} if for every nonsingular form $h'$ of rank $k$ over $(A,\s)$ we have
$\sign_P^\eta h \geq \sign_P^\eta h'$. 

\item We call a  hermitian form $h$  over $(A,\s)$ (resp. an element $u \in \sas$) 
\emph{$\eta$-maximal at $P$} if $h^\ns$ (resp.  $\qnd{u}_\s$)  is $\eta$-maximal at $P$.

\end{enumerate}
\end{defi}

Observe that $m_P$ does not depend on the choice of $\eta$.

\begin{prop}\label{rank-sign} 
Let $P\in X_F$
and let   
\[M_P:= \max\{ \sign_P^\eta h \mid h \textrm{ is a rank $1$ nonsingular hermitian form over $(A,\s)$}\}.\]
Then 
\begin{enumerate}[$(i)$]
\item $ \max \{ \sign_P^\eta h \mid h \textrm{ is a rank $t$ nonsingular hermitian form over $(A,\s)$}\}  = t M_P$;
\item $m_P =\ell M_P$.
\end{enumerate}
\end{prop}

\begin{proof} If $P\in \Nil[A,\s]$, then $m_P=M_P=0$, so we may assume that $P \in \wt X_F$.

 $(i)$ Let $h$ be a nonsingular form of rank~$t$. Since $h$ is an orthogonal sum of forms of rank~$1$, $\sign_P^\eta h\leq tM_P$. The equality follows by taking  a form $h_0$ of rank~$1$ such that $\sign_P^\eta h_0= M_P$ and considering $t\x h_0$.

$(ii)$ The inequality $m_P\leq \ell M_P$ follows from the fact that a form of dimension~$1$ has rank~$\ell$ and thus is
an orthogonal sum of $\ell$ hermitian forms of rank~$1$.
For the other inequality, we now construct a form of dimension~$1$ and signature $\ell M_P$. 

Using the notation introduced in Section~\ref{Morita},  
the tuple $\eta$ of reference forms for $(A,\s)$ obviously behaves as follows under the equivalences 
in \eqref{diagram}:
\begin{equation*}
\xymatrix{
\eta \ar@{|->}[r] & f_*(\eta) \ar@{|->}[r]   & (s\circ f_*)(\eta) \ar@{|->}[r]   & (g\circ s \circ f_*)(\eta),
}
\end{equation*}
where $\ve=1$ since $P \in \wt X_F$, cf. Lemmas~\ref{rescue} and \ref{eps}. 
Since signature and rank are preserved under Morita equivalence (cf. \cite[Theorem~4.2]{A-U-prime} and \cite[\S 2.2]{BP1}), there exists a  form $\qf{d}_\vt$ of rank $1$ over $(D,\vt)$ such that $\sign_P^{(g\circ s \circ f_*)(\eta)} \qf{d}_\vt =M_P$. Let $w = \diag(d,\ldots, d) \in M_\ell(D)$ and consider the form $\qf{w}_{\vt^t}$. 
Then \eqref{diagram} yields  
forms $ \qf{f^{-1}(\Phi w)}_\s$ and $ \qf{\Phi w}_{ \ad_\Phi }$ such that
\begin{equation*}
\xymatrix{
\qf{f^{-1}(\Phi w)}_\s \ar@{|->}[r] &  \qf{\Phi w}_{ \ad_\Phi } \ar@{|->}[r]    & \qf{w}_{\vt^t} \ar@{|->}[r]  & \ell\x \qf{d}_\vt
}
\end{equation*}
(note that $s(\qf{u}_{ \ad_\Phi }) :=
\Phi^{-1} \qf{u}_{ \ad_\Phi }=  \qf{\Phi^{-1} u}_{\vt^t }$ for $u\in M_\ell(D)$, which is easy to check). 
Then, by \cite[Theorem~4.2]{A-U-prime},  
\[ \sign_P^\eta \qf{f^{-1}(\Phi w)}_\s = \ell  \sign_P^{(g\circ s\circ f_*)(\eta)} \qf{d}_{\vt}=   \ell M_P.\qedhere\]
\end{proof}

\subsection{The case of invertible elements}

Let $b_1,\ldots, b_t \in F^\x$. We use the notation $\pf{b_1,\ldots, b_t}:= \qf{1, b_1} \ox \cdots \ox \qf{1, b_t}$ 
for Pfister forms  and  also write
\[H(b_1,\ldots, b_t) :=\{P \in X_F \mid b_1,\ldots, b_t \in P\}\] 
for the corresponding Harrison set. Note that such Harrison sets form a basis of the Harrison topology on $X_F$.

\begin{thm}\label{main_thm_1} 
Let  $b_1,\ldots, b_t \in F^\x$,  $\pi=\pf{b_1,\ldots, b_t}$, $Y=H(b_1,\ldots, b_t)$ and $\eta$ be a tuple of reference forms for $(A,\s)$.
Assume that $a\in\sas^\x$ is $\eta$-maximal at all  $P \in Y$. Let $u\in \Sym(A,\s)^\x$. The following statements are equivalent:
\begin{enumerate}[$(i)$]
\item $u$ is $\eta$-maximal at all $P\in Y$.
\item $u \in \das (2^s \x \pi \ox \qf{a}_\s)$ for some $s \in \N$.
\end{enumerate}
\end{thm}

\begin{proof} Assume $(i)$. It follows from the assumptions that $\sign_P^\eta \qf{a,-u}_\s=0$ for all $P\in Y$. Hence 
$\sign_P^\eta (\pi\ox \qf{a,-u}_\s )= \sign_P \pi\cdot \sign_P^\eta \qf{a,-u}_\s  =0$ for all $P\in X_F$. 
Thus $\pi\ox \qf{a,-u}_\s$ is  torsion in $W(A,\s)$ by \cite[Theorem~4.1]{LU1}. In other words, there exists $s \in\N$ such that $2^s \x \pi\ox \qf{a,-u}_\s=0$ in $W(A,\s)$ by \cite[Theorem~5.1]{Sch1}, from which $(ii)$ follows. 

Assume $(ii)$, i.e. assume that  $u \in \das (h)$, where $h=2^s \x \pi \ox \qf{a}_\s$. Then $u=h(x,x)$ for some $x\in M=A^{r}$, where $r= 2^{s+t}$. Since $u$ is invertible, a standard argument shows that $M= xA \oplus (xA)^{\perp_h}$. Thus
\[h\simeq \qf{u}_\s \perp h',\]
for some hermitian form $h'$ over $(A,\s)$ of rank $\ell (2^{s+t}-1)$ (since $A\cong M_\ell (D)$, for some $F$-division algebra $D$).
By assumption we have for every $P\in Y$ that
\begin{equation}\label{eggs}
\sign_P^\eta h =  2^{s+t} m_P = \sign_P^\eta \qf{u}_\s + \sign_P^\eta h'. 
\end{equation}
Since $\sign_P^\eta \qf{u}_\s\leq m_P$ and 
$ \sign_P^\eta h' \leq \frac{m_P}{\ell} \rk (h')= m_P (2^{s+t}-1)$  (by Proposition~\ref{rank-sign}),
these inequalities are in fact equalities by \eqref{eggs}, and $(i)$ follows.
\end{proof}

\begin{remark} If $P\in \Nil[A,\s]$, then the statement ``$u$ is $\eta$-maximal at $P$'' is trivially true. 
Thus Theorem~\ref{main_thm_1}$(i)$ only needs to be checked for $P \in Y\cap \wt{X}_F$.
\end{remark}

\subsection{The general case}

The following result is the equivalent of Theorem~\ref{main_thm_1} when $u$ is not necessarily invertible.

\begin{prop}\label{four}  
Let  $b_1,\ldots, b_t \in F^\x$,  $\pi=\pf{b_1,\ldots, b_t}$, $Y=H(b_1,\ldots, b_t)$ and $\eta$ be a tuple of reference forms for $(A,\s)$.
Assume that $a\in\sas^\x$ is $\eta$-maximal at all  $P \in Y$.  Let $h$ be a hermitian form over $(A,\s)$.
The following statements are equivalent:
\begin{enumerate}[$(i)$]
\item $h^\ns$ is $\eta$-maximal  at all $P\in Y$.
\item $h^\ns$ is a subform of   $2^k \x \pi \ox \qf{a}_\s$ for some $k \in \N$.
\end{enumerate}
\end{prop}

\begin{proof} $(i)\Rightarrow (ii)$: We write $h \simeq h^\ns \perp 0$ and let $r:=\rk(h^\ns)$. Let $P\in Y$.
By Proposition~\ref{rank-sign} it follows that 
$\sign_P^\eta h^\ns = r m_P/\ell$.
Note that  $\sign_P^\eta \qf{a}_\s=m_P$ and that $\rk (\qf{a}_\s) = \ell$. 
It follows that $\sign_P^\eta (r\x \qf{a}_\s - \ell \x h^\ns)=0$ for every $P \in Y$. Therefore, 
by Pfister's local-global principle (\cite[Theorem~4.1]{LU1}, \cite[Theorem~5.1]{Sch1}), there exists $k\in \N$ such that $2^k \ell \x \pi\ox h^\ns \simeq 2^k r\x \pi\ox \qf{a}_\s$ and the result follows.

$(ii)\Rightarrow (i)$: Let $P \in Y$. By the assumption on $a$ and Proposition~\ref{rank-sign}, 
$2^k \x \pi \ox \qf{a}_\s$ is $\eta$-maximal. 
The conclusion follows by the additivity of $\sign_P^\eta$. 
\end{proof}

It follows from Proposition~\ref{three} and Proposition~\ref{four} that

\begin{thm}\label{main_thm_2}  
Let  $b_1,\ldots, b_t \in F^\x$,  $\pi=\pf{b_1,\ldots, b_t}$, $Y=H(b_1,\ldots, b_t)$ and $\eta$ be a tuple of reference forms for $(A,\s)$.
Assume that $a\in\sas^\x$ is $\eta$-maximal at all  $P \in Y$.  Let $u\in\Sym(A,\s)$.
The following statements are equivalent:
\begin{enumerate}[$(i)$]
\item $u$ is $\eta$-maximal  at all $P\in Y$.
\item $u \in \das (2^k \x \pi \ox \qf{a}_\s)$ for some $k \in \N$.
\end{enumerate}
\end{thm}

To conclude this section we consider $(A,\s)=(M_n(F), t)$, where $t$ denotes transposition, and   obtain  a result   similar to  a classical theorem of Gondard and Ribenboim \cite[Th\'eor\`eme~1]{G-R}:

\begin{cor} A  symmetric matrix over $F$ is positive semidefinite  
at all $P \in X_F$ if and only
if it is a sum of hermitian squares in $(M_n(F), t)$.
\end{cor}

\begin{proof}
We may take $\eta=(\qf{1}_t)$ as a tuple of reference forms for $(A,\s)$ since $\sign_P^\eta \qf{1}_t=n$ for every $P\in X_F$. Note that $\wt X_F =X_F$.
Let $U\in \Sym(M_n(F), t)$. Then
$U$  is positive semidefinite  at all  $P \in X_F$ if and only if  all nonzero eigenvalues of $U$ are positive at all $P\in X_F$ 
if and only if $\qnd{U}_t$  is $\eta$-maximal at all  $P \in X_F$.  
Finally, by Theorem~\ref{main_thm_2} with $a=1$ and $Y=H(1)=X_F$,  this happens if and only if
$U$ is a sum of hermitian squares in $(M_n(F), t)$.
\end{proof}

\section{A theorem and a question of Procesi and Schacher}

Procesi and Schacher already considered a notion of positivity of elements in an algebra with involution and proved
a result characterizing totally positive elements (in their sense) in terms of weighted sums of squares of symmetric elements, cf. \cite[Theorem~5.4]{P-S}. They also raised the question  of whether positive elements are always 
sums of hermitian squares (and not necessarily squares of symmetric elements), cf.  \cite[p.~404]{P-S}. In this spirit, 
after  showing how their notion of positivity relates to ours, we prove a sums of hermitian squares 
version of \cite[Theorem~5.4]{P-S}, using  Theorem~\ref{main_thm_2}, and use our techniques to fully answer the question raised in \cite[p.~404]{P-S} of whether positive elements are always 
sums of hermitian squares.

Let $(A,\s)$ be an $F$-algebra with involution, let $u\ in \Sym(A,\s)$.  In \cite{P-S},
Procesi and Schacher define the positivity of $u$ in terms of  the corresponding scaled involution trace form $\tasu$. 
Consider 
\[T_{(A,\s)}: A\x A\to K,\ (x,y) \mapsto \Trd_A(\s(x)y)\quad \text{for } x,y\in A\] 
and
\[T_{(A,\s,u)}: A\x A\to K,\ (x,y) \mapsto \Trd_A(\s(x)uy)\quad \text{for } x,y\in A,\] 
where $u\in \Sym(A,\s)$. These forms are both symmetric bilinear over $F$ if $\s$ is of the first kind and hermitian 
over $(K,\iota)$  if $\s$ is of the second kind. 
The first form is always nonsingular, whereas the second form is nonsingular if and only if $u$ is invertible, 
cf. \cite[\S 11]{BOI}. 

Recall the following definitions from \cite[Definitions~1.1 and 5.1]{P-S}:

\begin{defi}\label{P-S-pos}
Let $P\in X_F$.
\begin{enumerate}[$(i)$]
\item The involution $\s$ is called \emph{positive at $P$} if  the form $\tas$ is positive semidefinite at $P$.
We also introduce the notation
\[X_\s:=\{P \in X_F \mid \s \text{ is positive at } P\}.\]
\item Assume that $\s$ is positive at $P$. An element $u\in \Sym(A,\s)$ is called \emph{positive at $P$} 
if the form $\tasu$ is positive semidefinite at $P$.
\end{enumerate}
\end{defi}

\begin{remark}\label{PD-PSD}
Recall that a nonsingular symmetric bilinear form over $F$ or a hermitian form over $(K,\iota)$ is positive semidefinite at a given ordering $P$  on $F$
if and only if it is positive definite at $P$.
\end{remark}

Another way of looking at the Procesi-Schacher notion of positivity is 
 from the point of view of signatures of involutions and signatures of
 hermitian forms, and specifically the signature of the form $\qf{u}_\s$.  
Propositions~\ref{eqvs2} and \ref{not_used} give the precise connections between these approaches, whereas Remark~\ref{eqvs1} describes
positivity of $u$ at $P$ in terms of a different trace form, $T_{(A,\s_u)}$, under a weaker hypothesis.

Recall from  \cite{LT} and \cite{Q} (or \cite[\S 11]{BOI})  that the signature of $\s$ at $P\in X_F$ is defined as
\begin{equation}\label{LTQ}
\sign_P \s :=\sqrt{\sign_P T_{(A,\s)}}.
\end{equation}

\begin{remark} If follows from \eqref{LTQ} that $\s$ is positive at $P\in X_F$ if and only if $\sign_P\s=\deg A (=n)$. 
\end{remark}

Recall that if $P \in \wt X_F$ then $A\ox_F F_P \sim D_P$, where 
$D_P$  is one of $F_P$, $F_P(\sqrt{-1})$ or $(-1,-1)_{F_P}$. We define
$\lambda_P = 1$ if $D_P = F_P$ or $F_P(\sqrt{-1})$ and $\lambda_P = 2$ if 
$D_P = (-1,-1)_{F_P}$. We also let  $n_P = n/{\lambda_P}$, so that $A\ox_F F_P
\cong M_{n_P}(D_P)$.

Now let $h$ be a  hermitian form  over $(A,\s)$ with adjoint involution $\ad_h$.
Then for $P \in X_F$,
\begin{equation}\label{h_adh}
\sign_P \ad_h = \lambda_P | \sign_P^\eta h |
\end{equation}
(if $P\in \Nil[A,\s]$, both sides of \eqref{h_adh} are zero), cf. \cite[Lemma~4.6]{A-U-Kneb}.  Note that the correspondence between $\ad_h$ and $h$ is unique only up to multiplication of $h$ by a nonzero element in $F$ and that $\lambda_P$ only depends on the Brauer class of $A$.

In the following proposition we collect a few elementary statements about signatures of involutions and one-dimensional forms.
For $u\in \sas^\x$ we write  $\s_u:=\Int(u^{-1})\circ \s$.

\begin{prop}\label{dontcare} 
Let  $u\in \sas^\x$ and let $P\in X_F$.
\begin{enumerate}[$(i)$]
\item $\sign_P \s_u =\lambda_P | \sign_P^\eta \qf{u}_\s|$.
\item $\sign_P\s_u \in \{0,\ldots, n\}$.
\item $\sign_P^\eta \qf{u}_\s \in \{-n_P, \ldots, n_P\}$. 
\item $\sign_P\s_u=n \Leftrightarrow |\sign_P^\eta \qf{u}_\s| =  n_P$.
\end{enumerate}
\end{prop}

\begin{proof} $(i)$ follows from \eqref{h_adh} since the involution $\s_u$ is adjoint to the form $\qf{u}_\s$, as can easily be verified. 

$(ii)$: Since $\dim_K A =m=n^2$ we have $\dim T_{(A,\s_u)} =m$. Using that  $\sign_P T_{(A,\s_u)}$ is always a square (cf. \cite{LT}, \cite{Q}) we obtain $\sign_P T_{(A,\s_u)} \in \{0, 1, 4,\ldots, (n-1)^2, n^2\}$ and thus $\sign_P\s_u \in \{0,\ldots, n\}$ by \eqref{LTQ}.

$(iii)$ follows from $(i)$ and $(ii)$, whereas $(iv)$ follows from $(i)$.
\end{proof}

\begin{remark}\label{rainbow} 
It is clear that $P \in X_\s$ if and only if the form $\tas$ is positive definite at $P$, cf. \eqref{LTQ}.
 Furthermore,
$m_P\leq n_P$ and  if $P\in X_\s$, then $m_P=n_P$ by Proposition~\ref{dontcare}$(iv)$. 
\end{remark}

As an immediate consequence of Proposition~\ref{dontcare} we obtain:

\begin{cor}\label{rainbow2}
The following statements are equivalent:
\begin{enumerate}[$(i)$]
\item $P \in X_\s$.
\item $|\sign_P^\eta \qf{1}_\s| =  n_P$ for all tuples of reference forms $\eta$.
\item $\sign_P^\eta \qf{1}_\s =  n_P$ for all tuples of reference forms $\eta$ of the form $(\qf{1}_\s,\ldots)$. 
\end{enumerate}
\end{cor}

\begin{remark}\label{rainbow3} 
Let $P\in X_\s$. By Corollary~\ref{rainbow}, $P\in \wt X_F$ and so $\ve_P=1$ by definition of signature. 
Hence $(A\ox_F F_P, \s\ox\id)\cong (M_{n_P}(D_P), \ad_{\Phi_P})$, for some matrix 
$\Phi_P\in \Sym ( M_{n_P}(D_P), \bbar^t)$. It follows from \cite[Lemma~3.10]{A-U-Kneb} and Corollary~\ref{rainbow2}
that $\sign \Phi_P=\pm n_P$, where $\sign$ denotes the Sylvester 
signature of hermitian matrices. In other words, $\Phi_P$ is positive definite or negative definite and, up to replacing $\Phi_P$
by $-\Phi_P$ (since $\ad_{\Phi_P}=\ad_{-\Phi_P}$) we may assume that $\Phi_P$ is positive definite.
\end{remark}

In the following result  we make the link between Procesi and Schacher's notion of positivity (statement $(ii)$; see also
Definition~\ref{P-S-pos})  and  signatures of hermitian forms.

\begin{prop}\label{eqvs2} 
Let $\eta$ be a tuple of reference forms for $(A,\s)$,  $P\in X_F$ and $u\in \sas^\x$. Assume that $\s$ is positive at $P$.
The following statements are equivalent:
\begin{enumerate}[$(i)$]
\item The involution $\s_u$ is positive at $P$.
\item The form $\tasu$ is positive definite or negative definite at $P$.
\item $u$ or $-u$ is $\eta$-maximal at $P$.
\end{enumerate}
\end{prop}

\begin{proof} 
By \cite[(11.1)]{BOI} the involution $\s_u \ox \prescript{\iota}{}\s$ corresponds to $\ad_{\tasu}$ under the isomorphism $A\ox_K \prescript{\iota}{} A \too \End_K(A)$, where $(\prescript{\iota}{}A, \prescript{\iota}{}\s)$  is the conjugate algebra with involution of $(A,\s)$. It follows from the definition of $\prescript{\iota}{}\s$ that $\sign_P \s =\sign_P \prescript{\iota}{}\s$ 
and from \cite[Remark~4.2]{A-U-Kneb} that
\begin{equation*}
\sign_P \ad_{\tasu} = \sign_P \s_u \cdot \sign_P \s.
\end{equation*}
From \cite{LT} and \cite{Q} we obtain that
\begin{equation*}
|\sign_P \tasu |=\sign_P \ad_{\tasu}.
\end{equation*}
These two equalities  prove the equivalence $(i)\Leftrightarrow (ii)$.  The equivalence $(i) \Leftrightarrow (iii)$ follows from 
Proposition~\ref{dontcare}$(iv)$ and the fact that $n_P=m_P$, since $\s$ is positive at $P$.
\end{proof}

\begin{remark}\label{eqvs1}
If we drop the assumption that   $\s$ is positive at $P$ in Proposition~\ref{eqvs2}, we obtain
(from \eqref{LTQ} and Proposition~\ref{dontcare}$(iv)$)
a
similar sequence of equivalences, but in terms of a different form, namely $T_{(A,\s_u)}$: 
let $\eta$ be a tuple of reference forms for $(A,\s)$, 
$P\in X_F$ and $u\in \sas^\x$. 
The following statements are equivalent:
\begin{enumerate}[$(i)$]
\item The involution $\s_u$ is positive at $P$.
\item The form $T_{(A,\s_u)}$ is positive definite at $P$.
\item $|\sign_P^\eta \qf{u}_\s|=n_P$.
\end{enumerate}
\end{remark}

The equivalence between $(ii)$ and $(iii)$ in Proposition~\ref{eqvs2} can be made more precise:

\begin{prop}\label{not_used}  
Let $\eta$ be a tuple of reference forms for $(A,\s)$,  $P\in X_F$ and $u\in \sas^\x$. Assume that $\s$ is positive at $P$. 
\begin{enumerate}[$(i)$]
\item If  ${1}$ is $\eta$-maximal at $P$, then $\tasu$ is positive definite at $P$ if and only if $u$ is $\eta$-maximal at $P$. 
\item If  ${-1}$ is $\eta$-maximal at $P$, then $\tasu$ is negative definite at $P$ if and only if $u$ is $\eta$-maximal at $P$.
\end{enumerate}
\end{prop}

\begin{proof} 
$(ii)$ follows from $(i)$ upon replacing $\eta$ by $-\eta$ and $u$ by $-u$. Thus, it suffices to prove $(i)$.

Observe that by Corollary~\ref{rainbow2} and Remark~\ref{rainbow}, $\s$ positive at $P$ implies that either $1$ or $-1$ is $\eta$-maximal at $P$. 
Also note that the assumption on $\s$ implies that $P\in \wt X_F$.  

Assume that ${1}$ is $\eta$-maximal at $P$. 
By Proposition~\ref{eqvs2} and since $T_{(A,\s, -u)} = -\tasu$, we only need to show the sufficient condition in $(i)$. Thus, assume that $u$ is $\eta$-maximal at $P$. 
 It is not hard to show that $\tasu \ox F_P = T_{(A\ox_F F_P, \s\ox \id, u\ox 1)}$. 
 We may therefore assume  that $F$ is real closed and, with reference to
Section~\ref{Morita}, we have $(A,\s)\cong (M_\ell(D), \ad_\Phi)$ for some
$\ell\in \N$, where $D$ is one of $F$, $F(\sqrt{-1})$ or $(-1,-1)_F$, equipped
with the conjugation involution $\bbar$ (which is the identity on $F$), and
$\Phi$ is some matrix  in $\Sym_\ve(M_\ell(D), \bbar^t)$. Observe that
$\ve = \ve_P$  and  $\ell=n_P$ since $F = F_P$, that $\ve_P = 1$ since $P \in \wt X_F$, and that $m_P=n_P$ since 
$P\in X_\s$.

Under the isomorphism  $(A,\s)\cong (M_\ell(D), \ad_\Phi)$, the element $u$ corresponds to a matrix $U \in \Sym(M_\ell(D), \ad_\Phi)^\x$, 
$\tasu$ corresponds to $T_{(M_\ell(D), \ad_\Phi, U)}$ and 
the tuple $\eta$ corresponds to a tuple $J$. 
By Remark~\ref{rainbow3} we may assume that $\Phi$ is positive definite. 
Since  $F$ is real closed, there exists an  invertible matrix $\Psi \in M_\ell(D)$ such that $\ovl{\Psi}^t=\Psi$ and  $\Phi = \Psi^2$. 

By  \eqref{diagram}, \eqref{seq} and the definition of signature, 
there exists $\delta \in \{-1,1\}$ such that for every matrix $B \in \Sym(M_\ell(D), \ad_\Phi)^\x$,
\[\sign^J \qf{B}_{\ad_\Phi} = \delta\sign (\Phi^{-1} B),\]
where $\Phi^{-1} B \in \Sym(M_\ell(D), \bbar^t)^\x$. 
By the asssumption on $1$, $\sign^\eta \qf{1}_\s>0$, which translates to 
$\sign^J \qf{I_\ell}_{\ad_\Phi}= \delta \sign (\Phi^{-1}) >0$, where $I_\ell$ denotes the $\ell\x \ell$ identity matrix. Since 
$\sign \Phi^{-1}  = \sign \Phi >0$, we deduce that $\delta=1$ so that $\sign^J \qf{B}_{\ad_\Phi}=\sign (\Phi^{-1} B)$.

By hypothesis $ \sign^\eta \qf{u}_\s=\ell$.   Thus, applying the above with $B=U$ yields $\Phi^{-1} U \in \Sym(M_\ell(D), \bbar^t)^\x$ and
\[\sign (\Phi^{-1} U)=\sign^J \qf{U}_{\ad_\Phi} = \sign^\eta \qf{u}_\s=\ell\] 
(cf. \cite[Theorem~4.2]{A-U-prime} for the second equality), and thus that $\Phi^{-1}U$ is positive definite. Therefore we can write $\Phi^{-1} U= \ovl{\Gamma}^t \Delta \Gamma$, where $\Gamma$ is invertible in $M_\ell(D)$ and $\Delta \in M_\ell(D)$ is a diagonal matrix with positive diagonal coefficients in $F=\Sym(D,\bbar)$. 

Finally, since $u$ is invertible, $\tasu$ is nonsingular and so in order to  
show that $\tasu$ is positive definite it suffices to show that $T_{(M_\ell(D), \ad_\Phi, U)} (X,X)\geq 0$ 
for every $X\in M_\ell(D)$. We have
\begin{align*}
T_{(M_\ell(D), \ad_\Phi, U)} (X,X) &= \Trd_{M_\ell(D)} (\ad_\Phi(X)U X)\\
&=\Trd_{M_\ell(D)} (\Phi \ovl{X}^t \Phi^{-1} U X)\\
&=\Trd_{M_\ell(D)} (\Psi^2 \ovl{X}^t \Phi^{-1} U X)\\
&=\Trd_{M_\ell(D)} (\Psi \ovl{X}^t \Phi^{-1} U X\Psi)\\
&=\Trd_{M_\ell(D)} ( (\ovl{X\Psi})^t \Phi^{-1} U X\Psi )\\
&=\Trd_{M_\ell(D)} ( (\ovl{X\Psi})^t  \ovl{\Gamma}^t \Delta \Gamma    X\Psi )\\
&=\Trd_{M_\ell(D)} ( (\ovl{\Gamma X\Psi})^t   \Delta (\Gamma    X\Psi) )\\
&=\Trd_{M_\ell(D)} ( \ovl{Y}^t   \Delta Y )\\
&\geq 0,
\end{align*}
where  $Y=\Gamma    X\Psi$ and  the inequality follows by  direct computation. 
\end{proof}

We record the next result for future use:

\begin{prop}\label{Brauer_pos} 
Let $(A,\s)$ be an $F$-algebra with involution such that $X_\s\not=\varnothing$.
Then there exists an $F$-linear involution $\tau$ on $D$, of the same type as $\s$, such that $X_\s \subseteq X_\tau$.
\end{prop}

\begin{proof} Write $(A,\s)\cong (M_{\ell}(D), \ad_\Phi)$ with $\vt$, $\ve$ and $\Phi$ as in Section~\ref{Morita}. 
Since $X_\s\not=\varnothing$, we have $\wt X_F\not=\varnothing$. We may therefore assume that $\ve=1$ by Lemmas~\ref{rescue} and \ref{eps} and thus that
$\vt$ is of the same type as $\s$.

Consider the hermitian form $\qf{1}_\s$. It corresponds to an $\ell$-dimensional hermitian form $\qf{a_1,\ldots, a_\ell}_\vt$ via the isomorphisms in \eqref{diagram}.  We show that $X_\s\subseteq X_\tau$, where $\tau$ is the involution
$\vt_{a_1}$ on $D$.

Let $P\in X_\s$.
Let $\eta$ be a tuple of reference forms for $(A,\s)$ of the form $(\qf{1}_\s,\ldots)$, cf. Remark~\ref{ricola}.
The assumption $\sign_P \s=n=\deg A$ is equivalent with $\sign^\eta_P \qf{1}_\s=n_P$ by Corollary~\ref{rainbow2}. 
Since the form $\qf{1}_\s$ corresponds to  $\qf{a_1,\ldots, a_\ell}_\vt$,  we have
$\sign_P^{(g\circ s \circ f_*)(\eta)} \qf{a_1,\ldots, a_\ell}_\vt =n_P$ by \cite[Theorem~4.2]{A-U-prime}. Since $\deg D=n/\ell$, the signature of a one-dimensional hermitian form over $(D,\vt)$ is bounded by $n_P/\ell$ (since such a form
gives rise to a matrix in $M_{n_P/\ell}(D_P)$ during the signature computation).
It follows that $\sign_P^{(g\circ s \circ f_*)(\eta)} \qf{a_i}_\vt =n_P/\ell$ for all $i\in \{1,\ldots, \ell\}$. By Corollary~\ref{rainbow2}, the involution $\vt_{a_i}$ on $D$ is positive at $P$ for all $i\in \{1,\ldots, \ell\}$. In particular, $P\in X_\tau$.
Observe that since $a_1\in \Sym(D,\vt)^\x$, the involution $\tau$ is of the same type as $\s$.
\end{proof}

\subsection{A theorem of Procesi and Schacher}

Recall that we have an isomorphism $f: (A,\s) \to (M_\ell(D), \Int(\Phi) \circ \vt^t)$. It induces an isomorphism
of $F_P$-algebras with involution
\[f\ox \id: (A\ox_F F_P,\s\ox \id) \to (M_\ell(D)\ox_F F_P , (\Int(\Phi) \circ \vt^t)\ox \id).\] 
Consider an isomorphism $\alpha_P: M_\ell(D)\ox_F F_P\to M_{n_P}(D_P)$ and let $\Int(\Psi_P) \circ \bbar^t$
be the involution on $M_{n_P}(D_P)$ that corresponds to the involution $(\Int(\Phi) \circ \vt^t)\ox \id$ 
 under $\alpha_P$, where $\Psi_P \in \Sym_{\ve_P}( M_{n_P}(D_P), \bbar^t)^\x$.
We also define $f_P= \alpha_P \circ (f\ox \id)$. 

Note that if $P\in X_\s$, then in particular $P \in \wt X_F$, and thus $\ve_P=1$ and $(D_P, \bbar)$ is one of $(F_P, \id)$, $(F_P(\sqrt{-1}), \bbar)$, or
$((-1,-1)_{F_P}, \bbar)$, cf. Section~\ref{sec:sign}.

\begin{lemma}\label{zero} 
Let $P\in X_\s$ and $u\in \Sym(A,\s)$. 
Then $T_{(A, \s, u)}$ is positive semidefinite at $P$ if and only if $T_{(M_{n_P}(D_P), \bbar^t, \Psi_P^{-1} f_P(u\ox 1))}$ is positive semidefinite at the unique ordering on $F_P$.
\end{lemma}

\begin{proof} Note that $\ovl{\Psi_P}^t=\Psi_P$. Since $\s$ is positive at $P$, 
we may assume by Remark~\ref{rainbow3} that
$\Psi_P$ is a positive definite matrix over $D_P$. 
 Thus $\Psi_P$ has a square 
root in $M_{n_P}(D_P)$ 
and we write $\Psi_P=\Omega_P^2$ with $\overline{\Omega_P}^t=\Omega_P$. The form $T_{(A, \s, u)}$ is positive semidefinite at $P$ if and only if it remains so 
over $F_P$. We have, for $x\in A\ox_F F_P$,
\begin{align*}
(T_{(A, \s, u)}\ox F_P) (x,x) &= T_{(A\ox F_P, \s\ox\id, u\ox 1)}(x,x)\\
&= \Trd_{A\ox F_P} \bigl((\s\ox \id) (x) (u\ox 1) x\bigr)\\
&= \Trd_{M_{n_P}(D_P)}(\Psi_P \ovl{f_P(x)}^t \Psi_P^{-1} f_P(u\ox 1) f_P(x))\\
&= \Trd_{M_{n_P}(D_P)}(\Omega_P^2 \ovl{f_P(x)}^t \Psi_P^{-1} f_P(u\ox 1) f_P(x))\\
&= \Trd_{M_{n_P}(D_P)}(\Omega_P \ovl{f_P(x)}^t \Psi_P^{-1} f_P(u\ox 1) f_P(x) \Omega_P)\\
&= \Trd_{M_{n_P}(D_P)}(\ovl{y}^t \Psi_P^{-1} f_P(u\ox 1) y)\\
&=T_{(M_{n_P}(D_P), \bbar^t, \Psi_P^{-1} f_P(u\ox 1))}(y,y),
\end{align*}
where $y= f_P(x) \Omega_P$. The statement follows.
\end{proof}

\begin{lemma}\label{one} 
Let $P\in X_\s$ and $u\in \Sym(A,\s)$. 
Then $\tasu$ is positive semidefinite at $P$ if and only of $T_{(M_\ell(D), \vt^t, \Phi^{-1} f(u))}$ is positive semidefinite at $P$.
\end{lemma}

\begin{proof}
 Let $P \in X_\s$. By Proposition~\ref{Brauer_pos} 
we may choose the involution $\vt$ on $D$ such that $P \in X_\vt$. In particular, $X_\vt\not=\varnothing$ and thus 
$\wt X_F\not=\varnothing$.
By Lemma~\ref{rescue} we have $\ve=1$, i.e. $\Phi
\in\Sym(M_\ell(D), \vt^t)$.
Let $ \Int(\Lambda_P)\circ \bbar^t$ be the involution on $M_{n_P} (D_P)$, corresponding to the involution $\vt^t \ox \id$
on $M_\ell(D)\ox_F F_P$ under the isomorphism $\alpha_P$, where $\Lambda_P$ is  some 
matrix  in $\Sym_\delta(M_{n_P}(D_P),  \bbar^t)^\x$. By Remark~\ref{rainbow3} we have $\delta=1$ since 
$P\in X_\vt=X_{\vt^t}$.
The map $\alpha_P$ induces
an isomorphism of algebras with involution
\begin{equation}\label{bubble1}
(M_\ell(D)\ox_F F_P, \vt^t \ox \id) \cong (M_{n_P} (D_P), \Int(\Lambda_P)\circ \bbar^t).
\end{equation}
Since $P \in X_\vt$ we may assume that $\Lambda_P$ is positive definite by Remark~\ref{rainbow3}. 
Using the isomorphisms $f$  and $\alpha_P$  we have 
\begin{align*}
(A\ox_F F_P, \s \ox \id) &\cong (M_\ell(D)\ox_F F_P,  \Int (\Phi \ox 1) \circ (\vt^t \ox\id) )\\
&\cong (M_{n_P}(D_P), \Int (\Phi_P) \circ \Int(\Lambda_P)\circ \bbar^t) \\
&= (M_{n_P}(D_P),  \Int (Z_P)\circ \bbar^t), 
\end{align*}
where $\Phi_P=\alpha_P(\Phi\ox 1)$ and $Z_P= \Phi_P \Lambda_P$. 
In other words, $f_P=\alpha_P \circ (f\ox \id)$ induces an isomorphism of $F_P$-algebras with involution
\begin{equation}\label{bubble2}
(A\ox_F F_P, \s \ox \id) \cong  (M_{n_P}(D_P),  \Int (Z_P)\circ \bbar^t).
\end{equation}
Since $P \in X_\s$, $Z_P$ is positive or negative definite  (cf. Remark~\ref{rainbow3})
and up to replacing $\Phi$ by $-\Phi$ we may assume it is positive definite. 
By Lemma~\ref{zero} and \eqref{bubble2}, $\tasu$ is positive semidefinite at $P$ if and only if
$T_{(M_{n_P}(D_P),\bbar^t ,Z_P^{-1} f_P(u\ox 1))}$ is positive semidefinite. By Lemma~\ref{zero} and \eqref{bubble1},
$T_{(M_\ell(D), \vt^t, \Phi^{-1}f(u))}$ is  positive semidefinite at $P$ if and 
only if $T_{(M_{n_P}(D_P), \bbar^t, \Lambda_P^{-1} \alpha_P ( (\Phi^{-1}f(u)) \ox 1)}$ is positive semidefinite. The 
statement follows since 
\begin{align*}
\Lambda_P^{-1} \alpha_P ( (\Phi^{-1}f(u)) \ox 1) &= \Lambda_P^{-1} \alpha_P ( (\Phi^{-1}\ox 1) (f(u) \ox 1))\\ 
&= \Lambda_P^{-1} \Phi_P^{-1} \alpha_P   (f(u) \ox 1)\\
&= Z_P^{-1} f_P   (u \ox 1).\qedhere
\end{align*}
\end{proof}

\begin{lemma}\label{two} 
With notation as in \eqref{iso} we have
\[T_{(M_\ell(D), \vt^t, \Phi^{-1} f(u))} \simeq \ell \x( T_{(D, \vt, u_1)} \perp \cdots \perp T_{(D, \vt, u_k)} \perp 0 \cdots \perp 0) \]
when $(D,\vt,\ve)\not=(F, \id_F, -1)$. 
\end{lemma}

\begin{proof} It follows from \eqref{iso} that $T_{(M_\ell(D), \vt^t, \Phi^{-1} f(u))} \simeq T_{(M_\ell(D), \vt^t, \diag(u_1,\ldots, u_k,0,\ldots,0)   )}$. The statement follows from a direct matrix computation starting from the canonical decomposition of $M_\ell(D)$ into simple $M_\ell(D)$-modules: $M_\ell (D) \cong 
\underbrace{D^\ell \oplus\cdots \oplus D^\ell}_{\ell \textrm{ copies}}$. 
\end{proof}

\begin{lemma} \label{lemma_x}
Assume that $\tas\simeq \la b_1,\ldots, b_{m}\ra_\iota$ with all $b_i \in F^\x$. Then 
\[X_\s=H(b_1,\ldots, b_{m}).\]
\end{lemma}

\begin{proof} It follows from Definition~\ref{P-S-pos}$(i)$ and \eqref{LTQ} that $P\in X_\s$ if and only if $b_i \in P$ for all $i=1,\ldots, m$. 
\end{proof}

We have now laid the ground work for proving our sums of hermitian squares version of \cite[Theorem~5.4]{P-S}:

\begin{thm}\label{pony}
Let $u\in \Sym(A,\s)$ and let $\tas\simeq \la b_1,\ldots, b_{m}\ra_\iota$ with all $b_i \in F^\x$.
The following statements are equivalent:
\begin{enumerate}[$(i)$]
\item $\qnd{u}_\s$ is $\eta$-maximal at all $P \in X_\s$, where $\eta$ is any tuple of reference forms for $(A,\s)$ of the form $(\qf{1}_\s, \ldots)$.
\item The form $\tasu$ is positive semidefinite at all  $P \in X_\s$.
\item $u \in \das (2^r \x \pf{b_1,\ldots, b_m} \ox \qf{1}_\s)$ for some $r \in \N$.
\end{enumerate}
\end{thm}

\begin{proof}  The equivalence between $(i)$ and $(iii)$ follows from Theorem~\ref{main_thm_2}.

$(iii)\Rightarrow (ii)$: Assume that 
\[u=   \sum_{e \in \{0,1\}^m} b^e     \sum_{i} \s(x_{i,e}) x_{i,e},\] 
where $b^e =b_1^{e_1}\cdots b_m^{e_m}$  and $x_{i,e}\in A$. Let $x\in A\setminus \{0\}$. Then
\[
\Trd_A (\s(x) u x ) = \sum_{e \in \{0,1\}^m} b^e     \sum_{i} \Trd_A(    \s(x_{i,e}x) x_{i,e}x )
\]
is nonnegative at all $P\in X_\s$ by definition of $X_\s$, \eqref{LTQ}, and Lemma~\ref{lemma_x}.

$(ii)\Rightarrow (i)$: The implication is trivially true if $X_\s=\varnothing$. 
Thus we assume $X_\s\not=\varnothing$. 
By Proposition~\ref{Brauer_pos} we may assume that $\vt$ is of the same type as $\s$  (in particular, $\ve=1$) and that 
$X_\s \subseteq X_\vt$. Let $\xi$ 
 be the tuple of reference forms for $(D,\vt)$, obtained from $\eta$ via the Morita equivalences in \eqref{diagram}. 
Let $P\in X_\s$. We have the following equivalences (with PD meaning positive definite and PSD meaning positive semi\-definite, as usual):
\begin{align*}
\tasu& \text{ is PSD at $P$} & \\
&\iff \text{$T_{(M_\ell(D), \vt^t, \Phi^{-1} f(u))}$ is PSD at $P$ [by Lemma~\ref{one}]}\\
&\iff \text{$T_{(D, \vt, u_i)}$ is PSD at $P$ for $i=1,\ldots, k$ [by Lemma~\ref{two} since $\ve=1$]}\\
&\iff \text{$T_{(D, \vt, u_i)}$ is PD at $P$ for $i=1,\ldots, k$ [since all $u_i$ are invertible]}\\
&\iff \exists \delta\in\{-1,1\} \text{ such that $\delta u_i$ is $\xi$-maximal at $P$  for $i=1,\ldots, k$}\\
& \qquad\qquad\qquad\qquad\qquad\qquad\qquad\qquad \text{[by Proposition~\ref{not_used} since $P\in X_\vt$]}\\
&\iff \exists \delta\in\{-1,1\} \text{ such that $\delta \qnd{u}_\s$ is $\eta$-maximal at $P$}.
\end{align*}
Assume for the sake of contradiction that $\delta=-1$. 
Thus 
\[P\in \{Q \in X_\s \mid -\qnd{u}_\s \text{ is $\eta$-maximal at }Q\},\]
 which is open in $X_F$ since the map
 $\sign^{\eta} \qnd{u}_\s:X_F \to \Z$ is continuous \cite[Theorem~7.2]{A-U-Kneb}. Therefore, there exist $c_1,\ldots, c_t \in F^\x$ such that $P\in H(c_1,\ldots, c_t) \subseteq \{Q \in X_\s \mid -\qnd{u}_\s \text{ is $\eta$-maximal at }Q\}$. 
Applying Theorem~\ref{main_thm_2}
with $Y=H(c_1,\ldots, c_t)$ and $a=1$  
then gives 
$-u \in \das  (2^s \x \pf{c_1,\ldots, c_t } \ox \qf{1}_\s)$ for some $s \in \N$. A trace computation as in the proof  of 
$(iii)\Rightarrow (ii)$ above 
then shows that the form $\tasu$ is negative semidefinite at $P$, contradiction.
\end{proof}

\subsection{A question of Procesi and Schacher}

Consider the following property:

\begin{description}
\item[(PS)] for every $u\in \Sym(A,\s)$, the form $\tasu$ is positive semidefinite at all  $P \in X_\s$ if and only if
$u \in \das (2^s \x  \qf{1}_\s)$ for some $s \in \N$.
\end{description}

In  \cite[p.~404]{P-S}, Procesi and Schacher, motivated by \cite[Theorem~5.4]{P-S},
 ask if property (PS) holds for all $F$-algebras with involution $(A,\s)$  
and give a positive answer for quaternion algebras 
\cite[Corollary~5.5]{P-S} and in the case where $X_\s=X_F$ 
\cite[Proposition~5.3]{P-S}. In \cite{K-U} an elementary counterexample is produced   to (PS) in general and some cases are studied where (PS) holds.
Our previous results yield a slight improvement on \cite[Proposition~5.3]{P-S}:  

\begin{cor}\label{canary} 
If $X_\s =\wt X_F$, then property (PS)  holds.
\end{cor}

\begin{proof} Let $u\in \sas$ and let $\eta$ be a tuple of reference forms for $(A,\s)$
of the form $(\qf{1}_\s,\ldots)$. Then $\tasu$ is positive semidefinite on 
$X_\s = \wt X_F$ if and only if $\qnd{u}_\s$ is $\eta$-maximal at all $P \in \wt{X}_F$ (and, trivially, on $X_F$) by 
Theorem~\ref{pony}, which in 
turn is equivalent to $u \in \das (2^s \x  \qf{1}_\s)$ for some $s \in \N$ by Theorem~\ref{main_thm_2}
 with $a=1$ and $Y=H(1)$
and because $1$ is $\eta$-maximal on $X_F$.
\end{proof}

Consider the following variation on property (PS), where we enlarge the set of orderings on which positivity is verified
from $X_\s$ to $\wt X_F$:

\begin{description}
\item[(PS')] for every $u\in \Sym(A,\s)$, the form $\tasu$ is positive semidefinite at all  $P \in \wt X_F$ if and only if
$u \in \das (2^s \x  \qf{1}_\s)$ for some $s \in \N$.
\end{description}

We can use property (PS') to reformulate the question of Procesi and Schacher and obtain a full characterization of those $F$-algebras with involution for which (PS') holds:

\begin{thm}\label{biscuit} 
Property \textup{(}PS'\textup{)} holds if and only if $\wt X_F = X_\s$.
\end{thm}

\begin{proof}
Assume that $\wt X_F =X_\s$. Then (PS) equals (PS') and the conclusion follows from Corollary~\ref{canary}. Conversely, assume that (PS') holds. Since $1\in \das (\qf{1}_\s)$, the form $T_{(A,\s,1)}$ is positive semidefinite on 
$\wt X_F$ by (PS') and, since  $T_{(A,\s,1)}$ is nonsingular, it is in fact positive  definite on 
$\wt X_F$. It follows from \eqref{LTQ}
 that $\s=\s_1$ is positive on $\wt X_F$, i.e. $\wt X_F = X_\s$.
\end{proof}

\section*{Acknowledgement}

We thank University College Dublin for having provided us with the challenging environment in which the research presented in this paper was carried out.


\begin{small}

\def\cprime{$'$}

\textsc{School of Mathematics and Statistics, University College Dublin, Belfield,
Dublin~4, Ireland} 

\emph{E-mail address: } \texttt{vincent.astier@ucd.ie, thomas.unger@ucd.ie}

\end{small}

\end{document}